\documentclass[12pt,a4paper]{article}
\usepackage[T1]{fontenc}
\usepackage[margin=3cm,footskip=1cm]{geometry}

\usepackage{amsmath,amssymb,amsthm,bm,mathtools,xspace,booktabs,url}
\usepackage{graphicx,etoolbox}
\usepackage{ dsfont }
\usepackage{bbold}
\usepackage{tikz}
\usepackage{ mathrsfs }
\usepackage[shortlabels]{enumitem}
\usepackage{xcolor}
\usepackage{color}

\newcommand{\dd}{\mathrm{d}}
\makeatletter
\global\let\tikz@ensure@dollar@catcode=\relax
\makeatother
\setlist{
  listparindent=\parindent,
  parsep=0pt,
}
\usepackage{amssymb}

\usepackage{hyperref}
\usepackage{cleveref}

\AtBeginEnvironment{thebibliography}{
  \small
  \setlength\itemsep{0.75em plus 0.5em minus 0.75em}%
}

\usepackage{caption}
\usepackage[raggedright]{titlesec}

\numberwithin{equation}{section}
\usepackage{amsfonts}

\theoremstyle{plain} 
\newtheorem{theorem}{Theorem}[section]

\theoremstyle{definition}

\usepackage[above,below,section]{placeins}

\usepackage[all]{onlyamsmath}

\usepackage{multirow}

\usepackage{xcolor}
\usepackage{color}

\usepackage{mathrsfs}
\usepackage{pdfpages}

\definecolor{darkmagenta}{rgb}{0.5,0,0.5}
\definecolor{darkgreen}{rgb}{0,0.6,0}
\definecolor{darkblue}{rgb}{0,0,0.6}
\definecolor{darkred}{rgb}{0.8,0,0}
\definecolor{mellow}{rgb}{.847, 0.72, 0.525}

\begin{document}

\title{An explicit solution to the Skorokhod embedding problem for double exponential increments}

\author{Giang Nguyen\thanks{The University of Adelaide, School of Mathematical Sciences, SA 5005, Australia, \texttt{giang.nguyen@adelaide.edu.au}} 
\and Oscar Peralta\footnote{The University of Adelaide, School of Mathematical Sciences, SA 5005, Australia, \texttt{oscar.peraltaguitierrez@adelaide.edu.au}}
}
\date{}

\maketitle
\begin{abstract}
  Strong approximations of uniform transport processes to the standard Brownian motion rely on the Skorokhod embedding of random walk with centered double exponential increments. In this note we make such an embedding explicit by means of a Poissonian scheme, which both simplifies classic constructions of strong approximations of uniform transport processes \cite{griego1971almost} and improves their rate of strong convergence \cite{gorostiza1980rate}. We finalise by providing an extension regarding the embedding of a random walk with asymmetric double exponential increments.
\end{abstract}
\section{Introduction}
\label{sec:intro}
Let $\{X_k\}_{k\ge 1}$ be an i.i.d. collection of random variables with finite second moment, defined in some arbitrary probability space. Define the random walk
\begin{equation}\label{eq:Sn1}S_i=\sum_{k=1}^i X_k,\quad i\ge 0.\end{equation}
Let $(\Omega, \mathds{P}, \mathscr{F}, \{\mathscr{F}_t\}_{t\ge 0})$ be a \emph{given} filtered and complete probability space and let $\mathcal{B}=\{B_t\}_{t\ge 0}$ be an $\mathscr{F}_t$--adapted Brownian motion defined on it. The Skorokhod embedding theorem/problem \cite{skorokhod1982studies} states that there exists an increasing sequence of $\mathscr{F}_t$-stopping times $\{T_i\}_{i\ge 1}$ such that for all $i\ge 1$,
\begin{equation}\label{eq:mainSkorrokhod1}S_i\stackrel{d}{=} B_{T_i},\end{equation}
\begin{equation}\label{eq:Skorokhodmoments1}\mathds{E}\left[T_i-T_{i-1}\right]=\mathds{E}\left[X_1^2\right]\quad\mbox{and}\quad\mathds{E}\left[(T_i-T_{i-1})^2\right]\le 4\mathds{E}\left[X_1^4\right],\end{equation}
where $T_0=0$. In other words, for any given random walk $\{S_i\}_{i\ge 1}$ one can find a sequence of epochs $\{T_i\}_{i\ge 0}$, called \emph{Skorokhod stopping times}, such that the Brownian motion $\mathcal{B}$ observed at those epochs is \emph{distributionally} equivalent to $\{S_i\}_{i\ge 1}$. A survey of methods to construct $\{T_i\}_{i\ge 0}$ such that (\ref{eq:mainSkorrokhod1}) holds can be found in \cite{obloj2004skorokhod}. Some of those methods rely on \emph{external randomization}, meaning that $\{T_i\}_{i\ge 1}$ might also depend in some mechanism independent of $\mathcal{B}$. For such a case, one needs to appropriately enlarge the original probability space, which is straightforward to do and in most situations implicit.

In this paper, we give a strikingly simple construction of $\{T_i\}_{i\ge 1}$ for the particular case in which $\{X_k\}_{k\ge 1}$ have a common density function on the form
\begin{equation}\label{eq:fX1}f(x)=\sqrt{\lambda} e^{-2\sqrt{\lambda}|x|},\quad x\in\mathds{R}, \lambda>0,\end{equation}
called the \emph{centered double exponential law} of parameter $2\sqrt{\lambda}$. Its associated random walk $\{S_i\}_{i\ge 0}$ and corresponding Skorokhod embedding have been intensively used in the literature concerning strong approximations of uniform transport processes to the standard Brownian motion and related processes, for instance, in \cite{griego1971almost,gorostiza1979strong,gorostiza1980rate,bardina2019strong,bardina2016complex}. More specifically, they were used to show the existence of a family of uniform transport processes $\mathcal{F}^\lambda=\{F^\lambda_t\}_{t\ge 0}$ on the form 
\begin{align}\label{eq:defFlambda1}
F^\lambda_t=\sqrt{\lambda}\int_0^t J_s^\lambda \dd s
\end{align}
that converge strongly to $\mathcal{B}$ as $\lambda\rightarrow\infty$,
where $\{J^\lambda_t\}_{t\ge 0}$ is a Markov jump process on $\{1,-1\}$ with initial distribution $(1/2, 1/2)$ and intensity matrix $\left(\begin{smallmatrix}-\lambda&\lambda\\ \lambda& -\lambda \end{smallmatrix}\right)$.

Here we prove that, with increments following (\ref{eq:fX1}), the Skorokhod stopping times $\{T_i\}_{i\ge 1}$ can be taken to correspond to the arrivals of a Poisson process  independent of $\mathcal{B}$. This draws a nice bridge between the existing uniform transport processes literature, and the theory of random models observed at Poissonian times such as the one recently developed for L\'evy processes \cite{albrecher2016exit,albrecher2017strikingly}. Even more, it allows for a simpler description of the construction of strong approximations of uniform transport processes to the standard Brownian motion in \cite{griego1971almost}, together with improved rates of convergence compared to those provided in \cite{gorostiza1980rate}. Additionally, we include how this embedding carries on to the case of Brownian motion with drift which, by means of \cite{griego1971almost}, allows us to specify a new family of strongly convergent transport processes.
\section{Main result}\label{sec:main}
Let $\{X_k\}_{k\ge 1}$  be an i.i.d. sequence of random variables whose elements follow a common law on the form (\ref{eq:fX1}), and let $\{S_i\}_{i\ge 0}$ be defined by (\ref{eq:Sn1}). We call $\{S_i\}_{i\ge 0}$ the \emph{centered double exponential random walk} of parameter $2\sqrt{\lambda}$.
\begin{theorem}[Skorokhod embedding for the centered double exponential r.w.]\label{th:main1}
Let $(\Omega, \mathds{P}, \mathscr{F}, \{\mathscr{F}_t\}_{t\ge 0})$ be a filtered probability space which supports an $\mathscr{F}_t$-adapted Brownian motion and an independent Poisson process of intensity $2\lambda$ with arrival times $\{\tau_i\}_{i\ge 0}$ (with $\tau_0:=0$). Then $T_i=\tau_i$ ($i\ge 0$) solves (\ref{eq:mainSkorrokhod1}) and (\ref{eq:Skorokhodmoments1}) in the case $\{S_i\}_{i\ge 0}$ is a centered double exponential random walk of parameter $2\sqrt{\lambda}$.
\end{theorem}
\begin{proof}
First, notice that both equations in (\ref{eq:Skorokhodmoments1}) concern moments of the exponential distribution, which translate as
\begin{align*}
&\mathds{E}\left[\tau_i-\tau_{i-1}\right]= \frac{1}{2\lambda} =\mathds{E}\left[X_1^2\right],\quad\mbox{and}\\
&\mathds{E}\left[(\tau_i-\tau_{i-1})^2\right]= \frac{2!}{2^2\lambda^2}\le \frac{4\cdot 4!}{2^4\lambda^2} = 4\mathds{E}\left[X_1^4\right],
\end{align*}
and thus (\ref{eq:Skorokhodmoments1}) holds. By the independent and stationary increments property of $\mathcal{B}$, it is enough to prove (\ref{eq:mainSkorrokhod1}) for $i=1$, which we do next. The Wiener--Hopf factorisation for the Brownian motion \cite[Corollary 2.4.10]{bladt2017matrix} implies that $B_{\tau_1}$ is equal in distribution to $Y_1-Y_2$, where $Y_1$ and $Y_2$ are independent exponential random variables of parameter $2\sqrt{\lambda}$. Thus, the density of $B_{\tau_1}$ at $x\in\mathds{R}$ is given by
\begin{align*}
f_{Y_1-Y_2}(x)& =\int_{-\infty}^\infty f_{Y_1}(x-z) f_{-Y_2}(z)\dd z\\
& = \int_{-\infty}^\infty \left(2\sqrt{\lambda}e^{-2\sqrt{\lambda} (x-z)}\mathds{1}\{x-z\ge 0\}\right)\left(2\sqrt{\lambda}e^{2\sqrt{\lambda} z}\mathds{1}\{z\le 0\}\right)\dd z\\
& = \sqrt{\lambda}e^{-2\sqrt{\lambda} x} \int_{-\infty}^\infty 4\sqrt{\lambda} e^{4\sqrt{\lambda} z}\mathds{1}\{ z\le 0\wedge x\}\dd z\\
& = \sqrt{\lambda}e^{-2\sqrt{\lambda} x} \left[e^{4\sqrt{\lambda}(x\wedge 0)}\right]=\sqrt{\lambda}e^{-2\sqrt{\lambda} |x|},
\end{align*}
implying that $B_{\tau_1}\stackrel{d}{=} S_1$ and thus the proof is finished.
\end{proof}
Having a precise description of the Skorokhod stopping times used in the construction of uniform transport process implies two things. The first is, some elements of these constructions can be simplified. For instance, the initial construction of \cite{griego1971almost} may be instead setup as a standard Brownian motion which is observed at Poisson times, avoiding the use of results concerning the Skorokhod embedding problem, which we sketch next.\footnote{From Oscar: my goal is to give a one paragraph construction. Is it clear enough?} Regard $\{\tau_i\}_{i\ge 1}$ as in Theorem \ref{th:main1}, define $\ell(0)=0$ and let $\{\ell(k)\}_{k\ge 1}$ be the inflection epochs of the random walk $\{B_{\tau_{i}}\}_{i\ge 0}$, that is,
\[\ell(k)=\inf\{i>\ell_{k-1}: (B_{\tau_{i}} - B_{\tau_{i-1}})(B_{\tau_{i+1}} - B_{\tau_{i}})<0\}.\]
Let $\theta_k=\tau_{\ell(k)}$. Theorem \ref{th:main1} implies that for all $k\ge 1$, $|B_{\theta_k}-B_{\theta_{k-1}}|$ follows the distribution of a $\mbox{Geo}(0.5)$--convolution of $\mbox{Exp}(2\sqrt{\lambda})$ distributions, which is itself an $\mbox{Exp}(\sqrt{\lambda})$ distribution. Moreover, a thinning argument implies that $\{\theta_k\}_{k\ge 1}$ is a Poisson process of parameter $\lambda$. Now, let $\mathcal{F}^\lambda=\{F^\lambda_t\}_{t\ge 0}$ be the continuous piecewise-linear process with $F^\lambda_0=0$, slopes $\pm \sqrt{\lambda}$, and whose inflection epochs $\{\chi_k\}_{k\ge 1}$ are such that $F^{\lambda}_{\chi_k}=B_{\theta_k}$ for all $k\ge 1$. Since $\chi_k-\chi_{k-1}=(\lambda)^{-1/2} |B_{\theta_k}-B_{\theta_{k-1}}|\sim\mbox{Exp}(\lambda)$, then $\{\chi_k\}_{k\ge 1}$ is also a Poisson process of parameter $\lambda$, implying that $\mathcal{F}^\lambda$ is indeed a uniform transport process as the one described by the r.h.s. of (\ref{eq:defFlambda1}). The precise arguments regarding its strong convergence to $\mathcal{B}$ are akin to those in \cite[p. 1131]{griego1971almost}.

The second implication is that results concerning the strong rate of convergence of uniform transport processes to the standard Brownian motion can be improved. A considerable amount of work in \cite{gorostiza1980rate} concerns quantitavely measuring how \emph{well} the random grid induced by the Skorokhod stopping times $\{T_i\}_{i\ge 1}$ approximates the deterministic grid $\{\mathds{E}(T_i)\}_{i\ge 0}$ on compact intervals. More specifically, they show that if $\{T_i\}_{i\ge 1}$ are Skorokhod stopping times  associated to the centered double exponential random walk of parameter $2n$ (for large $n\in \mathds{N}$), then
\begin{equation}\label{eq:tauspeed1}\mathds{P}\left(\max_{1\le i \le 2n^2}\left|T_i-\mathds{E}(T_i)\right|>\beta(q)\delta_n\right)=o(n^{-q})\quad \mbox{with}\end{equation}
\[\delta_n = n^{-1} (\log n)^{4+3/(4 \log n)}.\]
To do so, they use Doob's $L_p$-maximal inequality, a bound on the moments of Skorokhod stopping times \cite{sawyer1974skorokhod} and combinatorial arguments \cite[Equation (5)]{gorostiza1980rate}. Since we have identified $\{T_i\}_{i\ge 0}$ as being Poisson arrival times $\{\tau_i\}_{i\ge 0}$ of intensity $2n^2$, we get the following shortened proof which additionally yields an improved rate of convergence.
\begin{theorem}
Let $\{\tau_i\}_{i\ge 1}$ be the arrivals of a Poisson process of parameter $2n^2$. Then, for each $q>0$ there exists some $\beta^*(q)>1$ such that
\begin{equation}\label{eq:tauspeed2}\mathds{P}\left(\max_{1\le i \le 2n^2}\left|\tau_i-\mathds{E}(\tau_i)\right|>\beta^*(q)\delta_n^*\right)=o(n^{-q})\quad \mbox{with}\end{equation}
\[\delta_n^* = n^{-1} \lfloor \log n\rfloor.\]
\end{theorem}
\begin{proof}
Doob's $L_p$-maximal inequality implies that for $n\in\mathds{N}\cap(e,\infty)$,
\begin{align*}
\mathds{P}\left(\max_{1\le i \le 2n^2}\left|\tau_i-\mathds{E}(\tau_i)\right|>\beta^*(q)\delta_n^*\right) & \le \frac{\mathds{E}\left[(\tau_{2n^2}-1)^{2\lfloor \log n\rfloor}\right]}{(\beta^*(q)\delta_n^*)^{2\lfloor \log n\rfloor}}.
\end{align*}
Since $\tau_{2n^2}\sim\mbox{Erlang}(2n^2,2n^2)$, then Lemma 4.5 of \cite{nguyen2019strong} implies that
\begin{align*}
\mathds{E}\left[(\tau_{2n^2}-1)^{2\lfloor \log n\rfloor}\right]&\le \frac{(2\lfloor \log n\rfloor)!\sqrt{2n^2}}{(2n^2)^{2\lfloor \log n\rfloor}}\frac{\sqrt{2n^2}^{2\lfloor \log n\rfloor+1}-1}{\sqrt{2n^2}-1}\\
&\le C_1\left(\frac{2\lfloor \log n\rfloor}{2n^2}\right)^{2\lfloor \log n\rfloor} n^{2\lfloor \log n\rfloor+2} = C_1 n^{2}\left(\frac{\lfloor \log n\rfloor}{n}\right)^{2\lfloor \log n\rfloor},
\end{align*}
where $C_1$ is some positive real constant. Thus,
\begin{align}
\mathds{P}\left(\max_{1\le i \le 2n^2}\left|\tau_i-i/(2n^2)\right|>\beta^*(q)\delta_n^*\right) & \le C_1n^2\left(\beta^*(q)^{-2 (\log n-1)}\right)\label{eq:boundaux1}.
\end{align}
By letting $\beta^*(q):=e^{2q + 2}$ the r.h.s. of (\ref{eq:boundaux1}) is $O(n^{-2q})$ and the proof is finished.
\end{proof}
By using the rate of convergence (\ref{eq:tauspeed1}) and further calculations, the authors in \cite{gorostiza1980rate} are able to prove that for every $q>0$ there exists some $\alpha(q)>0$
\[\mathds{P}\left(\sup_{s\in[0,1]}\left|F^{(2n^2)}_s-B_s\right|>\alpha(q) n^{-1/2}(\log n)^{5/2}\right)=o(n^{-q}).\]
If instead one takes $T_i=\tau_i$ ($i\ge 0$) and plugs the rate of convergence (\ref{eq:tauspeed2}) in their proof, one gets that there exists some $\alpha^*(q)>0$ such that
\[\mathds{P}\left(\sup_{s\in[0,1]}\left|F^{(2n^2)}_s-B_s\right|>\alpha^*(q) n^{-1/2}\log n\right)=o(n^{-q}),\]
which corresponds to the same rate of strong convergence obtained in \cite{nguyen2019strong} for a related construction which is shown to converge to the Markov-modulated Brownian motion.

\section{Extension}
Theorem \ref{th:main1} can be easily extended to the case of asymmetric double exponential random walks as follows.
\begin{theorem}\label{th:asym1}
Let $\{\tau_i\}_{i\ge 0}$ be a Poisson process of intensity $2\lambda$, let $\{B_t\}_{t\ge 0}$ be an independent standard Brownian motion and fix $\mu\in\mathds{R}, \sigma > 0$. For $t\ge 0$ define $W_t=\mu t + \sigma B_t$. Then, for all $i\ge 1$
\[S^{\mathrm{a}}_i\stackrel{d}{=}W_{\tau_i},\]
where $\{S^{\mbox{a}}_i\}_{i\ge 0}$ is a random walk whose increments are driven by the density function
\begin{equation}\label{eq:asym1}f^{\mathrm{a}}(x)= \left[\frac{\omega_\lambda}{\omega_\lambda+\eta_\lambda}\right]\eta_\lambda e^{-\eta_\lambda x}\mathds{1}\{x\ge 0\} + \left[\frac{\eta_\lambda}{\omega_\lambda+\eta_\lambda}\right]\omega_\lambda e^{\omega_\lambda x}\mathds{1}\{x< 0\},\end{equation}
where
\[\eta_\lambda=\sqrt{\frac{\mu^2}{\sigma^4} + \frac{4\lambda}{\sigma^2}}- \frac{\mu}{\sigma^2}\quad\mbox{and}\quad\omega_\lambda=\sqrt{\frac{\mu^2}{\sigma^4} + \frac{4\lambda}{\sigma^2}}+\frac{\mu}{\sigma^2}.\]
\end{theorem}
\begin{proof}
Let $Y_1^{\mathrm{a}}$ and $Y_2^{\mathrm{a}}$ be the Wiener--Hopf factors of a Brownian motion with drift \cite[Corollary 2.4.10]{bladt2017matrix}, in the sense that $W_{\tau_1}\stackrel{d}{=}Y_1^{\mathrm{a}}-Y_2^{\mathrm{a}}$ with $Y_1^{\mathrm{a}}\perp Y_2^{\mathrm{a}}$, $Y_1^{\mathrm{a}}\sim\mbox{Exp}(\eta_\lambda)$ and $Y_2^{\mathrm{a}}\sim\mbox{Exp}(\omega_\lambda)$. Since
\begin{align*}
f_{Y_1^{\mathrm{a}}-Y_2^{\mathrm{a}}}(x)& = \int_{-\infty}^\infty \left(\eta_\lambda e^{-\eta_\lambda (x-z)}\mathds{1}\{x-z\ge 0\}\right)\left(\omega_\lambda e^{\omega_\lambda z}\mathds{1}\{z\le 0\}\right)\dd z\\
& = \frac{\eta_\lambda\omega_\lambda}{\omega_\lambda+\eta_\lambda} e^{-\eta_\lambda x} \left[e^{(\omega_\lambda+\eta_\lambda)(x\wedge 0)}\right],
\end{align*}
then (\ref{eq:asym1}) follows.
\end{proof}
Using Theorem \ref{th:asym1}, Section \ref{sec:main} and \cite[p. 1131]{griego1971almost}, it is straightforward to construct strongly convergent transport processes $\mathcal{G}^\lambda=\{G^\lambda_t\}_{t\ge 0}$ (to $\{W_t\}_{t\ge 0}$ as $\lambda\rightarrow\infty$) on the form
\[G^\lambda_t=\int_0^t \frac{2\lambda}{\eta_\lambda}\mathds{1}\{K^\lambda_s=+\} - \frac{2\lambda}{\omega_\lambda}\mathds{1}\{K^\lambda_s=-\} \dd s,\]
where $\{K^\lambda_t\}_{t\ge 0}$ is a Markov jump process on $\{+,-\}$ with initial distribution $\left(\frac{\omega_\lambda}{\omega_\lambda+\eta_\lambda}, \frac{\eta_\lambda}{\omega_\lambda+\eta_\lambda}\right)$ and intensity matrix $\frac{2\lambda}{\omega_\lambda+\eta_\lambda}\left(\begin{smallmatrix}-\eta_\lambda& \eta_\lambda\\ \omega_\lambda& -\omega_\lambda \end{smallmatrix}\right)$.

\bibliographystyle{abbrv}
\bibliography{oscar}
\end{document}